\documentclass[11pt,twoside,reqno,psamsfonts]{amsart}
\usepackage[utf8]{inputenc}

\usepackage[left=3.5cm,top=3cm,right=3.5cm]{geometry} 
\usepackage[colorlinks = true, citecolor = black, linkcolor = black, urlcolor = black, pdfstartview=FitH]{hyperref}
\usepackage[pdftex]{graphicx}
\usepackage[utf8]{inputenc}
\usepackage{microtype, float}
\usepackage{amsmath, verbatim}
\usepackage{amsfonts}
\usepackage{amssymb}
\usepackage{epstopdf}
\usepackage{epsfig}
\usepackage{mathscinet}

\numberwithin{equation}{section}

\theoremstyle{plain}
\newtheorem{thm}{Theorem}[section]

\newtheorem{lemma}[thm]{Lemma}
\newtheorem{prop}[thm]{Proposition}

\theoremstyle{definition}

\theoremstyle{remark}
\newtheorem{remark}[thm]{Remark}

\newcommand{\R}{\mathbb{R}}

\newcommand{\C}{\mathbb{C}}

\newcommand{\eps}{\varepsilon}

\renewcommand{\epsilon}{\varepsilon}
\renewcommand{\rho}{\varrho}

\newcommand{\Sh}{\mathbb{S}}

\DeclareMathOperator{\arcosh}{arcosh}

\DeclareMathOperator{\InjRad}{InjRad}

\DeclareMathOperator{\SO}{SO}

\title{$L^p$ Norms of eigenfunctions on regular graphs and on the sphere}

\author{Shimon Brooks}
\address{Department of Mathematics, Bar-Ilan University, Ramat-Gan, 5290002 Israel}
\email{brookss@math.biu.ac.il}

\author{Etienne Le Masson}
\address{School of Mathematics, University of Bristol, University Walk, Bristol, BS8 1TW, UK}
\email{etienne.lemasson@bristol.ac.uk}

\thanks{E.L.M. was supported by the Marie Sk{\l}odowska-Curie Individual Fellowship grant 703162.  S.B. was supported by Israel Science Foundation grant 1119/13.}

\begin{document}

\maketitle 
\begin{abstract}
We prove upper bounds on the $L^p$ norms of eigenfunctions of the discrete Laplacian on regular graphs. We then apply these ideas to study the $L^p$ norms of joint eigenfunctions of the Laplacian and an averaging operator over a finite collection of algebraic rotations of the $2$-sphere. Under mild conditions, such joint eigenfunctions are shown to satisfy for large $p$ the same bounds as those known for Laplace eigenfunctions on a surface of non-positive curvature.
\end{abstract}

\section{Introduction}
Let $(M,g)$ be an $n$-dimensional compact Riemannian manifold and $\Delta$ the Laplace-Beltrami operator. If $u$ is an eigenfunction, $\Delta u = \lambda^2 u$, Sogge proved \cite{Sog88} that
$$ \| u \|_p \lesssim \lambda^{\sigma(n,p)} \|u\|_2,$$
with 
\begin{equation}\label{e:Sogge}
\sigma(n,p) = 
\begin{cases}
n\left(\frac12 - \frac1p\right) -\frac12 & \text{ if } \frac{2(n+1)}{n-1} \leq p \leq \infty\\
\frac{n-1}2\left(\frac12 - \frac1p \right) & \text{ if } 2 \leq p \leq \frac{2(n+1)}{n-1}.
\end{cases}
\end{equation}
These bounds are saturated on the sphere, in the case of high $p$ by the zonal spherical harmonics (concentration around points) and for low $p$ by Gaussian beams (concentration on closed geodesics).

Since the pioneering work of Sogge, the study of the $L^p$ norms of eigenfunctions has attracted a lot of interest as a way to understand how the geometry and the dynamics of the underlying space influences the shape of the eigenfunctions.
In non-positive curvature, Berard showed \cite{Berard} that the sup norm bound can be improved by a logarithmic factor,
and more recently Hassel and Tacy \cite{HT} showed that this improvement can be maintained for all $L^p$ norms with $p > \frac{2(n+1)}{n-1}$, yielding
\begin{equation}\label{e:log improv}
 \| u\|_p \lesssim_p \frac{\lambda^{\sigma(n,p)}}{\sqrt{\log \lambda}} \|u\|_2.
\end{equation}
The field has expanded in many directions; see also eg. \cite{SoggeZelditch}, \cite{KochTataruZworski}, \cite{BurqGerardTzvetkov}, and \cite{HR16} for related results and their generalizations.

In this article, instead of Riemannian manifolds, we first consider the setting of regular graphs. These graphs can be seen as discrete analogues of hyperbolic surfaces and present 
similar characteristics to negatively curved manifolds. Our main result in this context is a bound on eigenfunctions of the discrete Laplacian (or adjacency matrix) analogous to \eqref{e:log improv}. We then apply our graph methods to eigenfunctions satisfying some symmetry conditions on the sphere. More precisely, we look at joint eigenfunctions of the Laplacian and an averaging operator over a finite set of algebraic rotations satisfying mild non-degeneracy conditions, and show that for large $p$ such eigenfunctions satisfy the same bounds (\ref{e:log improv}) as in the non-positive curvature case.

\subsection{Eigenfunctions on regular graphs}

Let $\mathcal G$ be a finite $q+1$-regular graph, and $\mathcal T_{q+1}$ its universal cover, the $q+1$ regular tree. We will denote by $|\mathcal G|$ the number of vertices of the graph. We are interested in the eigenfunctions of the operator $T_q$ defined on $L^2(\mathcal G)$ by
$$T_q f(x) = \frac1{\sqrt{q}} \sum_{d(x,y) = 1} f(y).$$
This operator is connected to the adjacency matrix $A$ and the Laplacian $\Delta$ of the graph by the formulas
$$ \Delta = \frac1{q+1} A  - \text{Id} = \frac{\sqrt{q}}{q+1} T_q - \text{Id}, $$
so the eigenfunctions of $T_q$ coincide with the eigenfunctions of $A$ and $\Delta$.

In order to state the condition we require on the graph $\mathcal G$, we define, for any function $f: \mathcal T_{q+1} \to \C$, the operator
$$ \tilde S_n(f)(x) = \frac1{q^{n/2}} \sum_{d(x,y) = n} f(y).$$
Up to normalization, this is the averaging operator over spheres of radius $n$. Its projection to the graph will be denoted by $S_n$. 
We assume that for any $\delta > 0$, we have
\begin{equation}\label{e:cycles}
\| S_n\|_{L^1\to L^\infty} \lesssim_\delta q^{(-1/2 + \delta)n}, \quad \text{for all } n \leq N.
\end{equation}
where $N$ will be taken as large as possible.  A natural assumption for us is that $N>c\log|\mathcal{G}|$ for some $c$ as $|\mathcal{G}|\to\infty$, as in this case Theorem~\ref{t:graph} gives a bound analogous to (\ref{e:log improv}).

\begin{remark}
Note that if we write
$$ \delta_x(y) = \begin{cases} 1 & \text{if } y = x \\
0 & \text{otherwise,}
\end{cases}$$
then $\|S_n \delta_x\|_\infty$ is controlled by the number of cycles of length $2n$ through $x$. Condition \eqref{e:cycles} is always satisfied if $N$ is the radius of injectivity $\InjRad(\mathcal G)$, or equivalently if $2N$ is equal to the girth, and in this case we can take $\delta = 0$.
As remarked in \cite{BL10}, if $\mathcal G$ is taken at random among the $q+1$-regular graphs with fixed number of vertices, then the condition is satisfied for $N = c \log | \mathcal G | $, for some constant $c > 0$, with a probability tending to $1$ as $|\mathcal G | \to \infty$.
\end{remark}

\begin{thm}\label{t:graph}
Let $q \geq 2$ and $\mathcal G$ be a $q+1$ regular graph satisfying \eqref{e:cycles} for some $N \geq \InjRad(\mathcal G)$, then for any eigenfunction $\psi$ of $T_q$ we have
$$ \| \psi \|_p  \lesssim_p \frac{\| \psi \|_2}{\sqrt{N}} , $$
for any $p > 2$. 
\end{thm}

\begin{remark}
If $\mathcal G$ is a random regular graph of fixed degree and number of vertices, the theorem gives that for any $p > 2$,
$$ \| \psi \|_p \lesssim_p \frac{\| \psi \|_2}{\sqrt{\log |\mathcal G |}} , $$
almost surely when $|\mathcal G| \to \infty$. However in the case of random regular graphs with a large fixed degree, much stronger results are shown in \cite{BHY16}, namely
$$ \| \psi \|_\infty \lesssim \frac{(\log |\mathcal G|)^\delta}{\sqrt{|\mathcal G|}} \|\psi\|_2, $$
almost surely when $|\mathcal G| \to \infty$, for some $\delta > 0$.  On the other hand, our result holds for {\em all} graphs satisfying Condition (\ref{e:cycles}), whereas the stronger bounds obtained from probabilistic methods can only be shown to hold almost surely.
\end{remark}

There has been a growing interest recently in the study of eigenfunctions of the Laplacian on regular graphs seen both as independent combinatorial objects (deterministic \cite{BL10, ALM} and random \cite{BHY16}) and as graphs arising from symmetries on manifolds \cite{BL}. A first result of delocalization of eigenfunctions was given 
in \cite{BL10}. Although it is not explicitly derived in the paper, the proof gives a bound on the sup norm of eigenfunctions, but the uniform bound obtained directly from there is only
$$ \| \psi \|_\infty \lesssim \frac{\| \psi \|_2}{N^{1/4}}. $$
Our methods are similar to \cite{BL10} --- based on harmonic analysis on regular trees --- and combine an adaptation of the argument of \cite{BL10} to some ideas of \cite{HT}. They allow us to prove a better bound and generalize it to all $L^p$ norms for any $p >2$.

\subsection{Eigenfunctions on the sphere}\label{ss:sphere}
Eigenfunctions of the Laplacian on the round $2$-sphere lie at the opposite extreme from hyperbolic surfaces:  here the dynamics are stable, the spectrum has large multiplicities, the property of quantum ergodicity fails \cite{CdV}, and the Sogge $L^p$ bounds (\ref{e:Sogge}) are achieved \cite{Sog86,SoggeBook}.  On the other hand, the large multiplicities allow for different bases that exhibit different properties, and in joint work with E. Lindenstrauss \cite{BLML} we showed that quantum ergodicity {\em does} hold for  bases of joint eigenfunctions of the Laplacian and an averaging operator over a finite set of rotations (stronger results for random bases were already proved in \cite{ZelditchRandomSphere} and \cite{VanderKam}).  This followed from  a quantum ergodicity result on large regular graphs (first proved in \cite{ALM}), by essentially using the graph structure induced from the free subgroup of rotations, to transfer techniques from the graph setting to the sphere $\Sh^2$.
In this sense, our result on the sphere is in the same spirit as \cite{BLML}, as we take advantage of the additional symmetries to transfer our $L^p$ methods on graphs to get the Hassel-Tacy bounds (\ref{e:log improv}) on the sphere.  We remark that in order to get a logarithmic improvement, we will have to make some assumptions on the set of rotations, most notably that they not get too close to each other (in a very quantitative way); this condition is satisfied when the rotations are algebraic (see Lemma~\ref{l:algebraicity}).

So for $M \geq 2$, let $g_1,\ldots, g_M$ be a finite set of rotations in $\SO(3)$ with algebraic entries, that generate a free subgroup. We further assume that the stabilizer of any $x\in\Sh^2$ is either trivial or cyclic--- i.e., its intersection with words of length $n$ does not grow exponentially in $n$--- in order to conform with Condition (\ref{e:cycles}).  We define the operator $T_q$ acting on $L^2(\Sh^2)$ by
$$ T_q f(x) = \frac1{\sqrt{q}} \sum_{j=1}^M \left( f(g_j x) + f(g_j^{-1} x)\right), $$
with $q = 2M -1$.

\begin{thm}\label{t:sphere}
Let $\psi_s$ be an $L^2$-normalized eigenfunction of the positive Laplacian $\Delta$ on the sphere $\Sh^2$ 
$$ \Delta \psi_s = s(s+1) \psi_s$$
Assume that $\psi_s$ is also an eigenfunction of an averaging operator $T_q$ for a finitely-generated free algebraic subgroup of $SO(3)$ such that every non-trivial stabilizer of $x\in\Sh^2$ is cyclic.  Then we have for $p > 6$
$$ \|\psi_s \|_p \lesssim_p \frac{s^{\frac12 - \frac2{p}}}{\sqrt{\log s}}.$$
\end{thm}

Note that this matches the best-known bounds for surfaces of non-positive curvature (\ref{e:log improv}), proved by Hassell and Tacy \cite{HT}.

\begin{remark}
In order to keep the exposition concise and self-contained, we will give a simplified argument that proves Theorem~\ref{t:sphere} for $p>8$.  The more general result for $p>6$ is identical, but requires Sogge's more refined argument from \cite{Sog86,SoggeBook} in place of the ``baby" version we include that appeals directly to Young's inequality.  We will indicate where Sogge's estimate can be put into the argument to prove the full stated version of Theorem~\ref{t:sphere}.
\end{remark}

The paper is organized as follows. In Section \ref{ramanujan}, we prove Theorem \ref{t:graph} in the case of Ramanujan graphs. These are graphs satisfying an additional condition on the spectrum of the adjacency matrix (a maximal spectral gap), which simplifies the proof. This allows us to show the main ideas of the argument, in a case with minimal technical difficulties. Moreover, the result is stronger for these graphs, as we only require a weakened version of Condition \eqref{e:cycles}, and is therefore interesting in its own right. The general proof of Theorem \ref{t:graph}, without the Ramanujan assumption, is addressed in Section \ref{untempered}. We then prove Theorem \ref{t:sphere} in Section \ref{sphere} by combining the ideas of the previous sections with estimates on convolutions with zonal spherical harmonics.

\vspace{.2in}
{\bf Acknowledgements.}  We thank Melissa Tacy for helpful discussions, and in particular for patiently walking us through some of the relevant background for Section~\ref{sphere}.

\section{The case of Ramanujan graphs}\label{ramanujan}
We  parametrize the spectrum of $T_q$ by $\lambda = 2\cos{\theta_\lambda}$, with
$$\lambda \in \left[-2\cosh\left(\frac{\log q}{2}\right),2\cosh\left(\frac{\log q}{2}\right)\right].$$ It can be divided into two parts: the \emph{tempered spectrum} is the part contained in the interval $[-2,2]$, and the \emph{untempered spectrum} is the part lying outside this interval (corresponding to complex values of $\theta_\lambda$).

In order to outline the main argument of the proof, we first consider the case of Ramanujan graphs. These are graphs with an empty untempered spectrum (except for the eigenvalue 
$2\cosh\left(\frac{\log q}{2}\right)$ corresponding to the constant eigenfunctions). In this case we can weaken the girth assumption \eqref{e:cycles} and only ask that there exist 
$\beta > 0$ such that
\begin{equation}\label{e:cycles2}
\| S_n\|_{L^1\to L^\infty} \lesssim q^{-\beta n}, \quad \text{for all } n \leq N.
\end{equation}
as in \cite{BL10}. Note that although the famous construction of Ramanujan graphs of \cite{LPS88} gives graphs with large girth, automatically satisfying Condition \eqref{e:cycles2} for $N \gtrsim \log |\mathcal G|$, this is not a general property (see \cite{Glas03}) and we therefore need the condition.

We will consider the operator $S_n \Pi$, where $\Pi$ is the orthogonal projection on the tempered spectrum. In the case of Ramanujan graphs we have $\Pi f(x) = f(x) - \bar f$, where $\bar f = \frac1{|\mathcal G|} \sum_{y\in \mathcal G} f(y)$. In the proof of the general theorem, we will need to subtract terms involving other untempered eigenfunctions and control their norm, which requires the stronger assumption \eqref{e:cycles}.

Note that Theorem~\ref{t:graph} holds trivially for the constant function, since
$$\frac{\|1\|_p}{\|1\|_2}=\frac{1}{|\mathcal{G}|}\cdot|\mathcal{G}|^{1/p}=|\mathcal{G}|^{1/p-1/2}$$
(In fact, one easily disposes with all untempered eigenfunctions, see Lemma~\ref{sup norm decay}.)  Thus we may assume that $\psi$ is orthogonal to constants; i.e., that $\Pi\psi=\psi$.

\begin{lemma}\label{l:lplq}
 Assuming \eqref{e:cycles2}, we have the estimate
$$ \| S_n \Pi \|_{L^p \to L^{p'}} = q^{-\beta_p n}, \quad \text{for all } n \leq N,$$
for any conjugate exponents $1 \leq p < 2 < p' \leq \infty$, i.e. such that $\frac1p + \frac1{p'} =1$.
\end{lemma}

\begin{proof}
Note that for any function $f$ such that $\|f\|_1 = 1$
\begin{align*}
\| S_n \Pi f \|_\infty &= \| S_n f - S_n \bar{f} \|_\infty \\
&\leq \| S_n f\|_\infty + \|S_n \bar{f} \|_\infty  \\
&\lesssim \| S_n f\|_\infty + q^{n/2}  \bar{f} \\
&\lesssim q^{-\beta n} + \frac{q^{n/2}}{|\mathcal G|}.
\end{align*}
We deduce that \eqref{e:cycles2} is satisfied if we replace $S_n$ with $S_n \Pi$, since $N < \log_q | \mathcal G|$.

On the other hand, it is well known that $S_n$ can be expressed as a polynomial in $T_q$, namely $S_n = q^{n/2} F(T_q, n)$, where $F(2\cos \theta, n) = \phi_\theta(n)$ is the spherical function
 \begin{equation}\label{e:spherical}
   \phi_\theta(n)=q^{-n/2}\left(\frac2{q+1}\cos (n \theta) + \frac{q-1}{q+1}\frac{\sin (n+1)\theta}{\sin \theta} \right)
 \end{equation}
Equivalently, the eigenfunctions of $T_q$ and $S_n$ coincide and the spectrum of $S_n$ is the set $\left\{ q^{n/2} \phi_{\theta_j}(n), j=1, \ldots, |\mathcal G| \right\}$, 
where $\{ 2 \cos(\theta_j) \}_{j=1,\ldots, |\mathcal G|}$ is the spectrum of  $T_q$. 
The spectrum of $S_n \Pi$ is obtained purely from parameters $\theta$ in the tempered spectrum of $T_q$, i.e. $\theta \in \R$, and is therefore bounded by $n$. We have
$$ \| S_n \Pi \|_{L^2 \to L^2} = O(n).$$
By interpolation we obtain
$$ \| S_n \Pi \|_{L^p \to L^{p'}} \lesssim_p q^{-\beta_p n}, \quad \text{for all } n \leq N.$$
for any conjugate exponents $\frac1p + \frac1{p'} =1$ with $p>2$, by setting eg. $\beta_p = \frac12\beta(\frac2p - 1)$.
\end{proof}

We now let  $\lambda = \cos\alpha$, with $\alpha \in [0,\pi]$, be the eigenvalue of the eigenfunction $\psi$. The idea of the proof is to define an operator that localizes spectrally near the eigenvalue $\lambda$ and to bound its $L^2 \to L^p$ norm for $p >2$. The bound will be obtained by a $TT^*$ argument similar to what was used in \cite{SoggeBook, HT}.

We define for $N$ even
$$W_{N,\alpha} = \sum_{n=1}^{N/2}\cos(2n\alpha) P_{2n}(T_q/2) \Pi$$
where $P_n(\cos \theta) = \cos(n\theta)$ is the Chebyshev polynomial of the first kind. We can see $P_n(T_q/2)$ as a discrete wave propagator as suggested in \cite{BL10}. The operator $W_{N,\alpha}$ is then reminiscent of the spectral cluster operators usually considered in the study of $L^p$ norms of eigenfunctions (see in particular \cite{HT}).

The spectral localization, or at least the emphasis of the eigenvalue $\lambda$ is based on the following lemma from \cite{BLML}, whose proof we reproduce for the convenience of the reader. 
\begin{lemma}\label{l:unitarity}
 For every $\alpha \in [0,\pi]$, $N \geq 10$,
 $$ \sum_{n=1}^N \cos(n\alpha)^2 \geq 0.3 N$$ 
\end{lemma}

\begin{proof}
We write
$$ S_T := \sum_{n=1}^T \cos(n\alpha)^2. $$
Then
\begin{align*}
\frac1T S_T &= \frac1T \sum_{n=1}^T \left( \frac12 + \frac14 e^{2in\alpha} + \frac14 e^{-2in\alpha} \right) \\
&= \frac{2T - 1}{4T} + \frac1{4T} \sum_{n=-T}^T e^{2in\alpha} \\
&= \frac{2T - 1}{4T} + \frac{\sin[(2T+1)\alpha]}{4T \sin\alpha}.
\end{align*}
Now 
$$ \min_{\alpha \in [0,\pi]} \frac{\sin[(2T+1)\alpha]}{4T \sin\alpha} = \min_{\alpha \in [0,\pi/2]} \frac{\sin[(2T+1)\alpha]}{4T \sin\alpha} \geq -\frac1{4T\sin(\pi/(2T+1))},$$
since $|\sin\alpha|$ is monotone increasing in $[0,\pi/2]$ and $\sin(2T+1)\alpha > 0$ for $0 < \alpha < \pi/(2T+1).$

Hence $$ \frac1T S_T \geq \frac1{4T} \left( 2T-1 - \frac1{\sin(\pi/(2T+1))} \right).$$
For $T\geq 10$ the right-hand side of the inequality is $\geq 0.3$  .
\end{proof}

Using Lemma \ref{l:unitarity} and the fact that $\Pi\psi=\psi$, we get that
\begin{align*}
W_{N,\alpha} \psi &=  \sum_{n=1}^{N/2} \cos(2n \alpha) P_{2n}(\cos \alpha) \psi \\
& = \sum_{n=1}^{N/2}  \cos^2(2n \alpha) \psi
\end{align*}
and therefore
$$\|W_{N,\alpha}\psi\| \geq 0.3 N \, \|\psi\|.$$

Now we can write
$$N \| \psi \|_p \lesssim \| W_{N,\alpha} \psi \|_p \leq \| W_{N,\alpha} \|_{L^2\to L^p} \| \psi \|_2.$$
Hence the proof of Theorem \ref{t:graph} is reduced to the following proposition.
\begin{prop}\label{p:cluster graph}
We have
$$\| W_{N,\alpha} \|_{L^2 \to L^p} = O(\sqrt{N}),$$
for any $ 2 < p \leq \infty$.
\end{prop}

To prove Proposition \ref{p:cluster graph} we first recall the following lemma as stated in \cite{BL10}.
\begin{lemma}\label{estimate} 
Fix a point $o \in \mathcal T_{q+1}$ and write $|x| = d(o,x)$ for any $x \in \mathcal T_{q+1}$.
Let $\delta_o$ be the $\delta$-function supported at $o$, and $n$ a positive even integer.  Then
\begin{align*}
P_n(T_q/2)\delta_o(x) & =  \left\{ \begin{array}{ccc} 0  & \quad & |x|  \text{ odd } \quad \text{or} \quad |x|>n\\ \frac{1-q}{2q^{n/2}} & \quad & |x|<n \quad \text{and} \quad |x| \text{ even } \\ \frac{1}{2q^{n/2}} & \quad & |x| = n \end{array}\right.
\end{align*}
\end{lemma}

We then have the estimate (see Corollary 1 in \cite{BL10})
\begin{lemma}\label{l:lplq bis}
Let $N$ be as in Condition \eqref{e:cycles}. For any conjugate exponents $1 \leq p < 2 < p' \leq \infty$ there exists $\beta > 0$ such that
$$\| P_n(T_q/2) \Pi \|_{L^p \to L^{p'}} \lesssim q^{-\beta n},$$
for any $n \leq N$.
\end{lemma}

\begin{proof}
The proof is identical to Corollary 1 in \cite{BL10} and based on the fact that we can write
$$ P_n(T_q/2) = \sum_{k=0}^{n/2-1} \frac{1-q}{2q^{n/2}} S_{2k} + \frac1{2q^{n/2}} S_n.$$
We then use the estimate of Lemma \ref{l:lplq}
\end{proof}

We are now ready to prove the proposition.

\begin{proof}[Proof of Proposition \ref{p:cluster graph}]
We use a $TT^*$ argument, or in other words the fact that
\begin{equation}\label{e:TTstar}
\| T \|_{L^2 \to L^p}^2 = \| T^*\|_{L^{p'} \to L^2}^2 = \| T T^* \|_{L^{p'}\to L^p}
\end{equation}
for conjugate exponents $\frac1p + \frac1{p'} = 1$.

First note that
$$W_{N,\alpha} W_{N,\alpha}^* = \sum_{n=1}^N \sum_{k=1}^N \cos(n\alpha) \cos(k\alpha) P_n(T_q/2) P_k(T_q/2)\Pi$$
Now we use Lemma \ref{l:lplq bis} in addition to the formula
$$ P_n P_k = \frac12( P_{n+k} + P_{|n-k|} )$$
satisfied by the Chebyshev polynomials, to get
\begin{align*}
\| W_{N,\alpha} W_{N,\alpha}^* \|_{L^{p'}\to L^p} &\leq \sum_{n=1}^N \sum_{k=1}^N \| P_n(T_q/2) P_k(T_q/2) \Pi \|_{L^{p'}\to L^p} \\
&= \frac12 \sum_{n=1}^N \sum_{k=1}^N \left( \| P_{n+k}(T_q/2) \Pi \|_{L^{p'}\to L^p} +  \|P_{|n-k|} (T_q/2) \Pi\|_{L^{p'}\to L^p} \right) \\
& \leq \frac12 \sum_{n=1}^N \sum_{k=1}^N q^{-\beta n} q^{-\beta k} + \frac12 \sum_{n=1}^N \sum_{k=1}^N q^{-\beta |n-k|} \\
&\leq O(1) + \sum_{n=1}^N \sum_{k=n}^N q^{-\beta(k-n)} = O(N)
\end{align*}
By \eqref{e:TTstar}, we deduce
$$\| W_{N,\alpha} \|_{L^2 \to L^p} = O(\sqrt{N}).$$
\end{proof}

Using Proposition \ref{p:cluster graph}, we finally obtain for any $2 < p \leq \infty$
$$ \| \psi \|_p \lesssim_p \frac1{\sqrt{N}} \| \psi \|_2.$$

\section{Proof of Theorem \ref{t:graph}}\label{untempered}
The proof for the general case is similar to the case of Ramanujan graphs. Here we need the stronger assumption \eqref{e:cycles}. Denote by $K_k(x,y)$ the kernel of the operator $S_k$. We can then rewrite \eqref{e:cycles} as
\begin{equation*}
\sup_{x,y} |K_k(x,y)| \lesssim_\delta q^{(-1/2+\delta)k}
\end{equation*}
for all $k\leq N$.
We denote $\Pi_\epsilon$ to be the projection to the part of the spectrum in $[-2-\epsilon, 2+\epsilon]$, that is in an $\epsilon$-neighborhood of the tempered spectrum.  
The goal is to bound $\|S_n\Pi_\epsilon\|_{L^1\to L^\infty}$ to prove the analogue of Lemma \ref{l:lplq}; Theorem~\ref{t:graph} holds trivially for untempered $\psi$ as we will see below in Lemma~\ref{sup norm decay}, like the case of the constant function in section~\ref{ramanujan} above.  In particular, we may safely assume for every $\epsilon>0$ that $\Pi_\epsilon\psi=\psi$.

More precisely, assumption (\ref{e:cycles}) means that the sup-norm of an untempered eigenfunction decays exponentially with $N$.  In fact--- and this is the key point--- any function in the {\em span} of untempered eigenfunctions will also have a sup-norm that decays exponentially:
\begin{lemma}\label{sup norm decay}
Let 
$$f(x) = \sum_j \phi_j(x)$$
where the $\{\phi_j\}$ are mutually orthogonal untempered eigenfunctions, each of eigenvalue $>2+\epsilon$ in absolute value; i.e., $\Pi_\epsilon f =0$.  Then there exists $\delta(\epsilon)>0$ such that
$$\|f\|_\infty \lesssim_\epsilon q^{-\delta(\epsilon) k} \|f\|_2$$
for every $k\leq N$ satisfying (\ref{e:cycles}).
\end{lemma}

\begin{proof}
Let $x$ be a point at which $|f(x)|=\|f\|_\infty$.  We may assume without loss of generality that $f(x)>0$, and that each $\phi_j(x)>0$ --- since otherwise we could omit the negative ones and increase $\|f\|_\infty$ while reducing $\|f\|_2$.

We first consider $k$ even and write
\begin{eqnarray}
S_k f (x) & =  &\sum_j \lambda_{j,k} \phi_j(x) \geq \lambda_{\epsilon, k}\sum_{j}\phi_j(x)\nonumber\\ 
& \geq & \lambda_{\epsilon,k} |f(x)|\label{large S_k action}
\end{eqnarray}
where $\lambda_{\epsilon,k}$ is the eigenvalue for the convolution operator $S_k$ acting on an untempered eigenfunction of $T_q$-eigenvalue $2+\epsilon$, and the $\lambda_{j,k}>\lambda_{\epsilon, k}$ denote the respective eigenvalues of $S_k$ on each untempered $\phi_j$.  Using the expression of the spherical function \eqref{e:spherical}, we have
$$\lambda_{\epsilon,k} =  \frac2{q+1}\cos (k \theta) + \frac{q-1}{q+1}\frac{\sin (k+1)\theta}{\sin \theta}$$ 
where $\theta$ is defined by the relation $2\cos(\theta) = 2+\epsilon$, i.e. $\theta = -i \arcosh(1+\eps/2)$.
We see that $\lambda_{\epsilon,k}$ grows exponentially in $k$, with exponent growth determined by $\epsilon$; i.e. 
$$\log_q{\lambda_{\epsilon, k}} = C(\epsilon)\cdot k + O(1)$$
Here we have taken $k\in 2\mathbb{Z}$ to avoid sign issues with the negative side of the spectrum--- for $k$ even, all $\lambda_{j,k}>0$ and the inequality $\lambda_{j,k}>\lambda_{\epsilon,k}\gtrsim q^{C(\epsilon)k}$ holds for all $j$.

On the other hand, we have 
$$||K_k(x,\cdot)||_2 \lesssim_\delta q^{k\delta/2}$$
by assumption (\ref{e:cycles}), and so we see from Cauchy-Schwarz
\begin{eqnarray*}
|S_kf(x)| & \leq & \|K_k(x,\cdot)\|_2 \|f\|_2\\
& \lesssim_\delta & q^{k\delta/2}\| f\|_2
\end{eqnarray*}
Combining this with (\ref{large S_k action}) gives
\begin{eqnarray*}
\lambda_{\epsilon, k}|f(x)| & \leq & q^{k\delta/2} \|f\|_2
\end{eqnarray*}

Choosing $\delta(\epsilon)<C(\epsilon)$ small enough that $k\delta(\epsilon)<\log_q{\lambda_{\epsilon,k}}$, we get
$$|f(x)| \lesssim_{\delta(\epsilon)} q^{-k\delta(\epsilon)/2}$$
as required.  Further adjusting $\delta(\epsilon)$ if necessary, we can guarantee that the inequality extends to odd values of $k$ as well.
\end{proof}

The above Lemma will allow us to handle situations where there are many untempered eigenfunctions, because any linear combination of them still has sup-norm that decays exponentially relative to the $L^2$-norm.  Now we need to exploit the $L^2$-norm in order to bound $S_n\Pi_\epsilon$ from $L^1\to L^\infty$.

\begin{lemma}\label{l2 l1}
For any function on the graph, we have
$$\|f\|_2 \leq \|f\|_1$$
\end{lemma}
\begin{proof}
Without loss of generality we may assume that $\|f\|_1=1$.  This means in particular that $|f(x)|\leq 1$ for all $x$ in the graph, and so $|f(x)|^2\leq |f(x)|$ for every $x$, yielding
\begin{eqnarray*}
\| f\|_2^2 & = & \sum_x |f(x)|^2\\
& \leq & \sum_x |f(x)|\\
& \leq & \|f\|_1 = 1
\end{eqnarray*}
Taking square-roots gives $\|f\|_2\leq 1$ as required.
\end{proof}

\begin{lemma}\label{l1 l infty}
Under the assumption (\ref{e:cycles}), we have that for every $\epsilon>0$ there exists $\delta(\epsilon)>0$ such that
$$ \|S_n \Pi_\epsilon\|_{L^1\to L^\infty} \lesssim_\epsilon q^{-n/4}$$
for $n\leq \delta(\epsilon) \cdot N$, where $N$ is from assumption (\ref{e:cycles}).
\end{lemma}

\begin{proof}
Consider $\|f\|_1=1$, and write $f= \Pi_\epsilon f + f_\text{untemp}$ as the decomposition of $f$ into an ``$\epsilon$-tempered" piece and a second piece consisting of all sufficiently untempered spectral components of $T_q$-eigenvalue greater than $2+\epsilon$ in absolute-value.  

We have
\begin{eqnarray*}
\|S_n \Pi_\epsilon f \|_\infty & = & \|S_n f - S_nf_\text{untemp}\|_\infty\\
& \leq & \|S_n f\|_\infty + \|S_n f_\text{untemp}\|_\infty\\
& \lesssim & q^{-n/4}\|f\|_1 + \|S_n f_\text{untemp}\|_\infty\\
& \lesssim & q^{-n/4} + \|S_n f_\text{untemp}\|_\infty
\end{eqnarray*}
by the assumption (\ref{e:cycles}) on the kernel which guarantees that
$$\|S_n\|_{L^1\to L^\infty} \lesssim q^{-n/4}$$

So it remains to show that 
$$\|S_n f_\text{untemp}\|_\infty \lesssim_\epsilon q^{-n/4}$$
But since $\|f\|_1=1$, Lemma~\ref{l2 l1} says that $\|f\|_2\leq 1$, which means in particular that $\|f _\text{untemp}\|_2\leq 1$.  And since $S_n$ preserves the eigenspaces, $S_n f_\text{untemp}$ also belongs to the kernel of $\Pi_\epsilon$, and thus we can apply Lemma~\ref{sup norm decay} to show that
\begin{eqnarray*}
\|S_n f_\text{untemp}\|_\infty & \lesssim_\epsilon & q^{-\delta(\epsilon) N} \|S_n f_\text{untemp}\|_2\\
& \lesssim_\epsilon & q^{-\delta(\epsilon) N} q^{n/2} \|f_\text{untemp}\|_2\\
& \lesssim_\epsilon & q^{n/2-\delta(\epsilon) N} 
\end{eqnarray*}
where we have used the fact that $\|S_n\|_{L^2\to L^2}$ is bounded\footnote{Here, a uniform spectral gap could improve the bounds a bit.} by $q^{n/2}$ and the claim follows as long as $n \leq \delta(\epsilon) N$.
\end{proof}

\vspace{.2in}

The rest of the argument is as follows.  Fix $p>2$, and $\epsilon(p)$ to be chosen later.  We assume $N\gtrsim \log|\mathcal{G}|$ and set $n=\delta(\epsilon(p))N$.  The factor $\delta(\epsilon(p))$ will appear in the final estimate for $\|\psi\|_p/\|\psi\|_2$.

We have
\begin{eqnarray*}
\|S_n \Pi_{\epsilon(p)}\|_{L^1\to L^\infty} & \lesssim_{\epsilon(p)} & q^{-n/4}\\
\|S_n \Pi_{\epsilon(p)}\|_{L^2\to L^2} & \leq & \lambda_{\epsilon(p), n} \leq q^{n\epsilon(p)}
\end{eqnarray*}
We pick $\epsilon(p)$ small enough that interpolation gives an exponential decay bound
$$\|S_n \Pi_{\epsilon(p)} \|_{L^{p'}\to L^p} \lesssim_{\epsilon(p)} q^{-0.01n}$$
for all $n < \delta(\epsilon(p)) N$, and thus in the $TT^*$ argument we will get a bound of 
$$\frac{\|\psi\|_p}{\|\psi\|_2} \lesssim_{\epsilon(p)} \frac{1}{\sqrt{\delta(\epsilon(p)) N}}\lesssim_p \frac{1}{\sqrt{N}}$$ 
for all eigenfunctions $\psi$, with implied constant determined by $p$.

\section{Joint Eigenfunctions on the Sphere}\label{sphere}

In this section we apply the above arguments to the case of joint eigenfunctions of the Laplacian and an averaging operator over a free group of algebraic rotations.  For simplicity, we will give a self-contained proof of the result for $p>8$; to extend to the full strength of Theorem~\ref{t:sphere}, one couples the argument below with Sogge's ``freezing" method to exploit stationary phase cancellations and reduce the estimate to a one-dimensional Young's inequality, instead of the wasteful two-dimensional Young's inequality we apply below (see note at equation (\ref{e:part2})).

We let $\mathcal{H}_s$ be the eigenspace of spherical harmonics of eigenvalue $s(s+1)$, and $Z_s\in\mathcal{H}_s$ the zonal spherical harmonic centered at the north pole $0$.  These have the reproducing property that 
\begin{equation}\label{reproducing}
Z_s\ast \psi = \psi
\end{equation}
for any $\psi\in\mathcal{H}_s$. We have $Z_s(x) = Z_s^\flat (d(0,x))$ for a function $Z_s^\flat$ on $[0,\pi]$. We will abuse notation and drop the extra decoration, and see $Z_s$ itself as a function on $[0,\pi]$. 

We will use  standard bounds on Legendre polynomials \cite[Theorem 7.3.3]{SzegoBook}, which yield the estimate
\begin{equation}\label{Z bound}
Z_s(t) \lesssim \min\{s, \sqrt{s/t}\}
\end{equation}

\subsection{$L^p$ bounds for spherical harmonics, $p>8$}
In this section we motivate our argument through a simplified proof of Sogge's bounds (\cite{Sog86}) for spherical harmonics, in the easier case $p>8$ where we may simply apply Young's inequality.  

The reproducing property (\ref{reproducing}), together with orthogonality of eigenspaces, means that convolution with $Z_s$ is the projection to $\mathcal{H}_s$.  Thus it is the norm of this convolution $\chi_s: f\mapsto Z_s\ast f$ as an operator from $L^2(\mathbb{S}^2)\to L^p(\mathbb{S}^2)$ that we wish to estimate.  Since this convolution is a projection, we could equally well estimate the norm from $L^{p'}(\mathbb{S}^2)\to L^p(\mathbb{S}^2)$ for $\frac{1}{p'} = 1-\frac1p$ and take the square root, since
\begin{eqnarray*}
\|\chi_s\|^2_{L^2\to L^p} & = & \|\chi_s\chi_s^*\|_{L^{p'}\to L^p}\\
& = & \|\chi_s\|_{L^{p'}\to L^p}
\end{eqnarray*}

\begin{lemma}[$L^p$ bounds for spherical harmonics, $p>8$]\label{l:young}
We have
$$\|\chi_s\|_{L^2\to L^p}\lesssim s^{\frac12-\frac2p}$$
\end{lemma}

\begin{proof}
Since we have explicit bounds on the kernel $Z_s$ of $\chi_s$, let us apply Young's inequality directly to estimate $\|\chi_s\|_{L^{p'}\to L^p}$:
\begin{eqnarray*}
\|\chi_s\psi\|_{L^p} & \leq & \|Z_s\|_{L^{p/2}}\cdot\|\psi\|_{L^{p'}}.
\end{eqnarray*}
The bound (\ref{Z bound}) gives
\begin{eqnarray}
\|Z_s\|_{L^{p/2}} & \lesssim & \left(\int_{0\leq t\leq s^{-1}}|Z_s(t)|^{p/2}tdt \right)^{2/p} + \left(\int_{s^{-1}\leq t\leq \pi}|Z_s(t)|^{p/2} tdt\right)^{2/p}\nonumber\\
& \lesssim & \left( s^{p/2}\cdot s^{-2}\right)^{2/p} + \left(\int_{s^{-1}\leq t\leq \pi} s^{p/4}t^{-p/4+1} dt		\right)^{2/p}\nonumber\\
& \lesssim & s^{1-\frac4p} + s^{1/2} \left(	t^{-\frac{p}{4}+2}\Big|_{s^{-1}}^\pi		\right)^{2/p}\nonumber\\
& \lesssim & s^{1-\frac4p} + s^{1/2}\left(	s^{\frac{p}{4}-2}	\right)^{2/p}\nonumber\\
& \lesssim & s^{1-\frac4p} + s^{1/2}s^{\frac{1}{2}-\frac4p}\nonumber\\
& \lesssim & s^{1-\frac4p} \label{p8}
\end{eqnarray}
as long as $p>8$.  
Substituting this back into Young's inequality above gives 
$$\|\chi_s\psi\|_{L^p}  \lesssim s^{1-\frac4p}\|\psi\|_{L^{p'}}$$
whereby 
$$\|\chi_s\|_{L^2\to L^p}^2 = \|\chi_s\|_{L^{p'}\to L^p} \lesssim s^{1-\frac4p}$$
as required.
\end{proof}

\subsection{The Improvement for Joint Eigenfunctions}
In order to apply the results of the previous sections, we wish to average together many copies of the kernel centered at different points determined by the averaging operator.  To this end, note that the essential contribution to the norm of $\chi_s$ comes from the $s^{-1+\delta}$ neighborhood of the center point; see precise statement in (\ref{e:kernel localization}) below.  Thus the contribution of points outside the radius of $s^{-1+\delta}$ is negligible, and we may safely combine copies of the kernel $Z_s$, centered at points that are $s^{-1+\delta}$-separated, without affecting the estimates.

Thus we must show that for our group of algebraic rotations, the image of any point under a word of length $n$ in the generators consists of points separated by a distance of $s^{-1+\delta }$.  This is where the algebraicity will come in, as in \cite{BLML}.  However, we cannot expect this to hold for all images at all points, since every rotation fixes a pair of antipodal points on the sphere, and so e.g. any fixed point of one of the generators will be fixed by powers of the generator, and thus will be fixed by (at least) $n$ words of length $\leq n$.  This is where we will need to use the assumption on the stabilizers, to ensure that we have only linear growth of the non-trivial stabilizers, and do not have exponential growth (which would violate Condition (\ref{e:cycles})).  

First, we need the following Lemma, essentially copied from \cite[Lemma 4.2]{BLML}:
\begin{lemma}\label{l:algebraicity}
There exists $c>0$ depending only on the generating set, such that for all $N<c\log{s}$ and all $x\in\mathbb{S}^2$, there exist at most $2$ words $g$ of length $N$ in our generators such that  $d(x,g.x)\leq s^{-1/4}$.  
\end{lemma}

As we mentioned above, we have to allow some ``bad" words, since if $x$ is stabilized by some word $g$, then powers of $g$ will also stabilize $x$.  Moreover a small neighborhood (eg., a ball of radius $s^{-1}$) of the fixed point will also stay close to itself under these powers of $g$.  But since the stabilizer of any such $x$ is cyclic, there can be only $2$ words of a given length, $g^n$ and $g^{-n}$, that fix $x$ and have the same axis of rotation.  

\begin{proof}[Proof of Lemma~\ref{l:algebraicity}]
Let $K$ be a finite extension of $\mathbb{Q}$ containing all entries of the matrices $g _ 1, \dots, g _ M$.
Recall the notion of (logarithmic) height of an algebraic number $\alpha \in K$ from  e.g.~\cite[Ch.~1]{Bombieri-Gubler}. What is important for us is that it is a nonnegative real number measuring the complexity of nonzero $\alpha \in K$ (e.g. for a rational number given in lowest terms by $p/q$ we have that $h(p/q)= \log \max (p,q)$) with the following properties:
\begin{itemize}
\item $h(\alpha \beta) \leq h (\alpha) + h (\beta)$
\item $h (\alpha + \beta) \leq h (\alpha) + h (\beta) + \log 2$ (cf. \cite[\S1.5.16]{Bombieri-Gubler})
\item for any embedding $\iota: K \hookrightarrow \mathbb{C}$ and any algebraic number $\alpha \neq 0$ we have that $ |{\iota (\alpha)}| \geq e^{-[K:\mathbb{Q}]h(\alpha)}$ (cf. \cite[\S1.5.19]{Bombieri-Gubler}).
\item $h (\alpha ^{-1}) = h (\alpha)$
\end{itemize}
It follows that if we set, for $g \in \SO (3, \overline{\mathbb{Q}})$, the height $h (g)$ of $g$ to be the maximum of the heights of its coordinates, then for any $g _ 1, g _ 2 \in \SO (3, \overline {\mathbb{Q}})$ we have
\begin{align*}
h (g _ 1 g _ 2) &\leq h (g _ 1) + h (g _ 2) + O (1) \\
h (g _ 1 ^{-1}) &\leq O (h (g _ 1) + 1)
.\end{align*}
Thus if $w (g _ 1, \dots, g _ M)$ is any word of length $N$ in the given generating set
\begin{equation*}
h (w (g _ 1, \dots, g _ M)) \lesssim N (1+\max \{h (g_1) ,\dots, h (g _ M )\})
.\end{equation*}
Since $g _ 1, \dots, g _ M$ generate a free group, it follows from this and the above basic property of heights, that for any reduced word $w$ of length $N$ if $g=w(g_1, \dots g _ M)$ then $ \|{g-1}\| \geq C^{-N}$ for some $C$ depending only on the generating set, in other words $g$ is a rotation of angle $\theta$ around its axis with 
$|{\theta}| \geq C ^ {- N}$.
Moreover, since the commutator of two words of length $N$ is a word of length at most $4N$, the commutators satisfy a similar bound (up to adjusting the constant), and therefore we deduce that the distinct axes of rotation of words of length $N$ are also separated by ${C}^{-4N}$.

Choosing $c$ small enough, we can guarantee that 
$${C}^{-4N} > {C}^{-4c\log{s}} = s^{4c\log{{C}}}> s^{-1/4}$$ 
and thus each word of length $\leq N$ in the generators is a rotation through an angle  $\geq s^{-1/4}$.  This means that $d(x,g.x)>s^{-1/4}$ for all points $x$ that are not close to a fixed point of $g$.  But by the bound on the commutators, we deduce that any point $x$ can be close to the fixed point of at most one axis of rotation, and due to the condition that all non-trivial stabilizers are cyclic, this means that $x$ can be close to the fixed point of at most $2$ words of length $N$, since any such words must be powers $g^n$ and $g^{-n}$ of the same element $g$.  
\end{proof}

\begin{proof}[Proof of Theorem~\ref{t:sphere}]
We keep the notations of the previous sections, and study the operator
$$W_{N,\alpha,s}= \sum_{n=1}^{N/2} \cos(2n\alpha)P_{2n}(T_q/2) \chi_s$$
where $\chi_s$ is the projection operator to the $s(s+1)$-eigenspace with kernel $Z_s$, the averaging operator $T_q$ is defined as in Section \ref{ss:sphere}, and $P_{2n}$ is the Chebyshev polynomials of the first kind of degree $2n$.  Since $\chi_s$ is the projection to $\mathcal{H}_s$, any joint eigenfunction $\psi\in\mathcal{H}_s$ of $T_q$-eigenvalue $2\cos\alpha$ will satisfy $\chi_s\psi=\psi$ and thus by Lemma~\ref{l:unitarity}
$$\|\psi\|_p \lesssim \frac1N\|W_{N,\alpha,s}\psi\|_p \leq \frac1N\|W_{N,\alpha,s}\|_{L^2\to L^p}\|\psi\|_2$$
and so it remains to show that 
$$\|W_{N,\alpha,s}\|_{L^2\to L^p} \lesssim \sqrt{N} s^{\frac12-\frac2p},$$
which will follow from an estimate of the type
\begin{equation}\label{e: sph lp' lp} 
\|S_n  \chi_s\|_{L^{p'} \to L^p} \lesssim s^{1-\frac{4}{p}} \, q^{-\beta_p n}
\end{equation}
for some $\beta_p >0$, by the arguments in Section~\ref{ramanujan}.  Naturally we will take $N\gtrsim_p \log{s}$, in a way which will depend on $p$.

To prove (\ref{e: sph lp' lp}), observe that the kernel $Z_s$ of $\chi_s$ is localized (in $L^{p/2}$) in the ball of radius $s^{-1/2}$, since 
\begin{eqnarray}
\left\|Z_s \Big|_{t\geq s^{-1/2}}\right\|_{L^{p/2}} & \lesssim & \left(\int_{s^{-1/2}\leq t\leq \pi}|Z_s(t)|^{p/2} tdt\right)^{2/p}\label{e:kernel localization}\\
& \lesssim &  \left(\int_{s^{-1/2}\leq t\leq \pi} s^{p/4}t^{-p/4+1} dt		\right)^{2/p}\nonumber\\
& \lesssim & s^{1/2} \left(	t^{-\frac{p}{4}+2}\Big|_{s^{-1/2}}^\pi		\right)^{2/p}\nonumber\\
& \lesssim & s^{1/2}\left(	s^{\frac{p}{8}-1}	\right)^{2/p}\nonumber\\
& \lesssim &  s^{1/2}s^{\frac{1}{4}-\frac2p}\nonumber\\
& \lesssim & s^{\frac34-\frac2p} \nonumber
\end{eqnarray}
Note that for $p>8$, this exponent $\frac34-\frac2p < 1-\frac4p$ ; i.e. it is of lower order than the $L^{p/2}$-norm of $Z_s$ inside the ball of radius $s^{-1/2}$, as computed in (\ref{p8}).

Recall now that
$$ \chi_s u (x) = Z_s * u (x) = \int Z_s(d(x,y)) u(y) \, dy.$$
 We split $Z_s$ into three parts, $Z_s^{(1)} + Z_s^{(2)} + Z_s^{(3)}$, with $Z_s^{(1)}(\theta) = \eta_s(\theta) Z_s(\theta)$ and $Z_s^{(3)}(\theta) = \eta_s(\pi - \theta) Z_s(\theta)$, where $\eta \in C_0^\infty(\R)$ is supported in $[0, s^{-1/4}]$ and is identically $1$ on $[0, s^{-1/2}]$ . We denote by $\chi_s = \chi_s^{(1)} + \chi_s^{(2)} + \chi_s^{(3)}$ the corresponding convolution operators.
Recall that we are interested in estimating the norm of $S_n \chi_s =S_n \chi_s^{(1)} + S_n\chi_s^{(2)} + S_n\chi_s^{(3)}$. The kernel of $S_n \chi_s$ is given by 
$$ [S_n \chi_s] (x,y) = \frac1{q^{n/2}} \sum_{|g| = n} Z_s(d(g x,y)),$$
so $S_n \chi_s$ is a normalized sum of convolutions with zonal spherical harmonics centered at different points.

We first treat the case of $S_n\chi_s^{(2)}$. Using Young's inequality, we want to compute:
$$ \left( \int \left| [S_n \chi_s^{(2)}] (x,y)\right|^{p/2} \, dx\right)^{2/p} = \left(\int \left| [S_n \chi_s^{(2)}] (x,y)\right|^{p/2} \, dy\right)^{2/p}, $$
which is bounded by
$$\frac1{q^{n/2}} \sum_{|g| = n} \|Z_s^{(2)}\|_{p/2} \leq q^{n/2} \|Z_s^{(2)}\|_{p/2} \leq q^{n/2} s^{\frac34-\frac2p} .$$
Now if $n \leq 2 \delta_p \log_q s$ we obtain by Young's inequality
\footnote{Here one inserts Sogge's more refined estimates for $\|Z_s^{(2)}\|_{L^2\to L^p}$ for all $p>6$ (see eg. \cite{Sog86,SoggeBook}) in order to deduce Theorem~\ref{t:sphere} for all $p>6$ as well.  }
\begin{equation}\label{e:part2}
\|S_n \chi_s^{(2)}\|_{L^{p'} \to L^p} \lesssim  s^{\frac34-\frac2p + \delta_p}.
\end{equation}
We choose $\delta_p$ so that $\frac34-\frac2p + \delta_p < 1 - \frac14$, which is possible as long as $p > 8$.

By symmetry, the cases of $S_n\chi_s^{(1)}$ and $S_n\chi_s^{(3)}$ are identical. Here we use the fact that for any given $y \in \Sh^2$ there are at most $2$ words $g$ of length $n$ such that $d(gx,y) \leq s^{-1/4}$, according to Lemma \ref{l:algebraicity}. This means we have at most $2$ terms contributing to the sum 
$$ [S_n \chi_s^{(1)}] (x,y) = \frac1{q^{n/2}} \sum_{|g| = n} Z_s^{(1)}(d(g x,y))$$
for each $x,y$, and so up to a factor of $2$ the kernel of $S_n\chi_s^{(1)}$ is a normalized average over a disjoint union of copies of $Z_s^{(1)}$
 \begin{eqnarray*}
\left( \int \left| [S_n \chi_s^{(1)}] (x,y)\right|^{p/2} \, dy\right)^{2/p} & \leq & \frac2{q^{n/2}} \left\|\sum_{|g| = n} Z_s^{(1)}(g\cdot) \right\|_{p/2}\\
&  \leq & 2q^{-n/2} \cdot q^{2n/p} \|Z_s^{(1)}\|_{p/2}\\
& \lesssim &  q^{n(\frac{4-p}{2p})} s^{1-\frac4p}.
\end{eqnarray*}
By Young's inequality we thus have
\begin{equation}\label{e:part1}
\|S_n \chi_s^{(1)}\|_{L^{p'} \to L^p} \lesssim q^{-(\frac{p-4}{2p}) n} s^{1-\frac4p},
\end{equation}
which gives the necessary estimate\footnote{In fact, this part of the estimate holds for all $p>4$; it is only the bound on the $S_n\chi_s^{(2)}$ term that determines the applicable values of $p$.}.

Putting \eqref{e:part1} and \eqref{e:part2} together and choosing $\delta_p$ small enough we obtain
$$\|S_n \chi_s\|_{L^{p'} \to L^p} \lesssim  q^{-\beta_p n} s^{1-\frac4p},$$
for some $\beta_p > 0$. As in Lemma~\ref{l:lplq bis}, this implies a similar bound 
$$\|P_{2n}(T_q/2)\|_{L^{p'}\to L^p}  \lesssim q^{-\beta_p n} s^{1-\frac4p}$$
and the argument of the proof of Proposition~\ref{p:cluster graph} carries through to show
$$\|W_{N,\alpha,s}\|_{L^2\to L^p} \lesssim \sqrt{N} s^{\frac12-\frac2p}$$
as required.

\end{proof}

\end{document}